\documentclass[12pt]{amsart}
\usepackage{amsmath,amsthm,amsfonts,amssymb,amscd,euscript}
\usepackage{color}

\newlength{\defbaselineskip}
\newcommand{\Gr}{{\rm Gr}}

\theoremstyle{plain}
\newtheorem{thm}{Theorem}

\newtheorem{lem}[thm]{Lemma}
\newtheorem{cor}[thm]{Corollary}

\newtheorem{prop}[thm]{Proposition}

\theoremstyle{definition}
\newtheorem{defn}[thm]{Definition}

\newtheorem{exmp}[thm]{Example}

\newtheorem{rem}[thm]{Remark}
\numberwithin{equation}{section}

\newcommand{\bigboxs}[1]
{ \multiput(#1)(40,0){2}
 {\line(0,40){40}}
\multiput(#1)(0,40){2}
 {\line(40,0){40}}
}

\allowdisplaybreaks[4]

\begin{document}

\title[Maps from quiver Grassmannians to ordinary Grassmannians]{On natural maps from strata of quiver Grassmannians to ordinary Grassmannians}
\author{Kyungyong Lee and Li Li}
\thanks{Research of the first author is partially supported by NSF grant DMS 0901367. }

\address{Department of Mathematics, Wayne State University, Detroit, MI 48202}
\email{{\tt klee@math.wayne.edu}}
\address{Department of Mathematics and Statistics, 
Oakland University,
Rochester, MI 48309}
\email{{\tt li2345@oakland.edu}}

\begin{abstract} Caldero and Zelevinsky \cite{CZ} studied the geometry of quiver Grassmannians for the Kronecker quiver and computed their  Euler characteristics by examining natural stratification of quiver Grassmannians. We consider generalized Kronecker quivers and compute virtual Poincare polynomials of certain varieties which are the images under projections from strata of quiver Grassmannians to ordinary Grassmannians. In contrast to the Kronecker quiver case, these polynomials do not necessarily have positive coefficients. The key ingredient is the explicit formula for noncommutative cluster variables given in \cite{ls-noncomm}. \end{abstract}

 \maketitle
 



\section{introduction}
Recently the study of \emph{quiver Grassmannians} has been very active [because of its important role in Cluster algebra, and the surprising fact that every projective variety is a quiver Grassmannian] (for instance, see \cite{CC, CK1, CK2, CZ, DWZ, FK, KQ, Q, R-every, r-quantum}).
Let $Q$ be a quiver on vertices $\{1,...,n\}$. A \emph{representation} of $Q$ is a family $M=(M_i, \varphi_a)$, where each $M_i$ is a finite dimensional $\mathbb{C}$-vector space attached to a vertex $i$, and each $\varphi_a$ is a linear map attached to an arrow $a:j\longrightarrow i$. The \emph{dimension vector} of $M$ is the integer vector \textbf{dim} $M=(\dim M_1,...,\dim M_n)$. A \emph{subrepresentation} of $M$ is an $n$-tuple of subspaces $N_i\subseteq M_i$ such that  $\varphi_a(N_j)\subseteq N_i$ for any arrow $a:j\longrightarrow i$. For every integer vector $\mathbf{e} = (e_1,...,e_n)$, the quiver Grassmannian $\Gr_\mathbf{e}(M)$ is defined as the variety of subrepresentations of $M$ with the dimension vector $\mathbf{e}$.

Much of the focus has been on acyclic quiver Grassmannians, where $Q$ has no oriented cycles. A very interesting result of \cite{N, Q,  Ef, KQ}
 shows that
if $M$ is an indecomposable rigid representation of an acyclic quiver $Q$, then the Euler characteristic of any quiver Grassmannian is nonnegative.
Caldero and Zelevinsky \cite{CZ} studied the geometry of quiver Grassmannians for the Kronecker quiver (the quiver with two vertices and two arrows from one vertex to the other) and obtained simple formulas for their  Euler characteristics by examining natural stratification of quiver Grassmannians.  More precisely, they showed that the Euler characteristics of $
Z_{p,s}(M)$ and $
Z'_{p,s}(M)$, which are defined below, are nonnegative.

Throughout the paper, let $Q$ be a (generalized) Kronecker quiver, i.e. a quiver with two vertices labeled 1 and 2, and with $r\geq 2$ arrows from 1 to 2.
Let $M=(M_1, M_2; \varphi_1,...,\varphi_r)$ be a representation of $Q$ of dimension vector $(d_1,d_2)$. Now we define the main geometric objects of the paper. Note that $Z_{p,s}(M)$ and $Z'_{p,s}(M)$ are studied in \cite{CZ}.

\begin{defn} For two nonnegative integers $p$ and $s$, we define the following sets
$$\aligned
&Z_{p,s}(M)=\{U\in \text{Gr}_s(M_1) \, : \, \dim (\sum_{k=1}^r \varphi_k(U))=d_2-p\},\\
&\overline{Z}_{p,s}(M)=\{U\in\Gr_s(M_1): \dim(\sum_{k=1}^r\varphi_k(U))\le d_2-p\},\\
&Z'_{p,s}(M)=\{U\in \text{Gr}_{d_2-s}(M_2) \, : \, \dim (\bigcap_{k=1}^r \varphi_k^{-1}(U))=p\},\\
&\overline{Z}'_{p,s}(M)=\{U\in\Gr_{d_2-s}(M_2): \dim(\bigcap_{k=1}^r\varphi_k^{-1}(U))\ge p\}.
\endaligned
$$
\end{defn}
Using the natural structure of subvarieties of the Grassmannians, $Z_{p,s}(M)$ and $Z'_{p,s}(M)$ are quasi-projective varieties (in this paper, a variety is allowed to be reducible), $\overline{Z}_{p,s}(M)$ and $\overline{Z}_{p,s}(M)$ are projective varieties, $\overline{Z}_{p,s}(M)=\bigcup_{p'\ge p}Z_{p',s}(M)$, $\overline{Z}'_{p,s}(M)=\bigcup_{p'\ge p}Z'_{p',s}(M)$.
Note that $\overline{Z}_{p,s}(M)$ may not be the closure of $Z_{p,s}(M)$ since it may not be irreducible, although there is no example where $\overline{Z}_{p,s}$ is reducible to the authors' knowledge.

For $n\ge0$, let $M(n+3)$ (resp. $M(-n)$) be the indecomposable rigid representation  with dimension vector $(c_{n+2}, c_{n+1})$ (resp. $(c_{n+1}, c_{n+2})$). For definition of $c_n$, see Definition~\ref{cn}. It is easy to see that the correspondence $U\mapsto U^\perp$ gives the following isomorphisms of varieties:
\begin{equation}\label{eq;ZisoZ'}
Z_{p,s}(M(-n))\cong Z'_{p,s}(M(n+3)),\quad \overline{Z}'_{p,s}(M(-n))\cong\overline{Z}_{p,s}(M(n+3)).
\end{equation}

In this paper we want to see if Caldero-Zelevinsky's approach can be applied to generalized Kronecker quivers. We compute  the Euler characteristics of a special class of $\overline{Z}'_{p,s}(M)$.    Unlike the case of the Kronecker quiver, the Euler characteristic of $\overline{Z}'_{p,s}(M)$ is shown to be negative  for some $p,s,M$. We actually compute a more refined invariant, namely the virtual Poincare polynomial of $\overline{Z}'_{p,s}(M)$, using the formula for noncommutative cluster variables given in \cite{ls-noncomm}. 

The virtual Poincare polynomial $P_Y\in\mathbb{Z}[q^{1/2}]$ can be assigned for every complex algebraic variety $Y$ that is not necessarily smooth, compact or irreducible (with a slight modification from \cite[p.198]{FMacP}). It has the following properties:

\noindent (i) if $Y$ is smooth and compact, then $P_Y(q)$ is the usual Poincare polynomial $\sum b_jq^{j/2}$ where $b_j$ is the rank of the $j$-th homology group of $Y$.

\noindent (ii) if $Z$ is a closed subvariety of $Y$ and $U$ is the complement of $Z$ in $Y$, then $P_Y=P_Z+P_U$.

\noindent (iii) if $Y'\to Y$ is a bundle with fiber $F$ which is locally trivial in Zariski topology, then $P_{Y'}=P_Y\cdot P_F$.


Now we state the main theorem and main example. Recall that $[m]_q=\sum_{i=0}^{m-1}q^i$, $[m]_q!=\prod_{i=1}^m[i]_q!$, and $\binom{m}{n}_q=[m]_q!/([n]_q![m-n]_q!)$.
\begin{thm}\label{mainthm2}
Let $r\ge 2$, and let $M(6)$ be the indecomposable rigid representation of $Q$ with dimension vector $(r^3-2r,r^2-1)$.

(1) For any nonnegative integer $e_1$,
$$P_{\Gr_{e_1,1}(M(6))}=\binom{r-1}{1}_q\binom{r-1}{e_1}_q + \sum_{j=1}^{r-1}q^{(r-e_1)j-1}\bigg{[} \binom{r}{1}_q\binom{r-1}{e_1}_q-\binom{r-j-1}{e_1-j}_q\bigg{]}.$$

(2) The virtual Poincare polynomial of  $\overline{Z}'_{p,r^2-2}$ is $$P_{\overline{Z}'_{p,r^2-2}(M(6))}=\sum_{e_1=p}^{d_1}P_{\Gr_{e_1,1}(M(6))}(-1)^{e_1-p}q^{\binom{e_1-p+1}{2}}\binom{e_1-1}{e_1-p}_q.$$
\end{thm}

\begin{exmp}\label{mainexample}
Let $r=10$. Then $\chi(\overline{Z}'_{5,98}(M(6)))=-27$ is negative. Indeed, its virtual Poincare polynomial can be computed using Theorem \ref{mainthm2}:
{\small
$$\aligned
&\phantom{++}q^{73}+2q^{72}+4q^{71}+7q^{70}+12q^{69}+19q^{68}+27q^{67}+36q^{66}+46q^{65}+55q^{64}+61q^{63}+63q^{62}\\
&+58q^{61}+46q^{60}+24q^{59}-5q^{58}-42q^{57}-81q^{56}-123q^{55}-158q^{54}-184q^{53}-195q^{52}-190q^{51}\\
&-164q^{50}-121q^{49}-62q^{48}+6q^{47}+77q^{46}+144q^{45}+198q^{44}+235q^{43}+249q^{42}+241q^{41}+209q^{40}\\
&+162q^{39}+101q^{38}+38q^{37}-25q^{36}-76q^{35}-116q^{34}-138q^{33}-146q^{32}-137q^{31}-119q^{30}-93q^{29}\\
&-65q^{28}-37q^{27}-14q^{26}+4q^{25}+15q^{24}+21q^{23}+22q^{22}+20q^{21}+16q^{20}+12q^{19}+8q^{18}+5q^{17}\\&+3q^{16}+2q^{15}+q^{14}+q^{13}+q^{12}+q^{11}+q^{10}+q^{9}+q^{8}+q^{7}+q^{6}+q^{5}+q^{4}+q^{3}+q^{2}+q+1.
\endaligned
$$
}
Note that the above example implies that for some $p'\ge 5$, $Z'_{p',98}(M(6))$ has negative Euler characteristic, and has a negative coefficient in its virtual Poincare polynomial.
A smaller example that also has a negative coefficient in the virtual Poincare polynomial (but with a positive Euler characteristic) is $\overline{Z}'_{1,23}(M(6))$ where $r=5$: indeed, using Theorem \ref{mainthm2},  $P_{\overline{Z}'_{1,23}(M(6))}$ is
$q^{22}+2q^{21}+2q^{20}+2q^{19}+q^{18}-q^{16}+q^{14}+2q^{13}+2q^{12}+2q^{11}+q^{10}+q^{9}+q^{8}+q^{7}+q^{6}+q^{5}+q^{4}+q^{3}+q^{2}+q^{1}+1$. As we see, $q^{16}$ has a negative coefficient.
\end{exmp}

The above computation has a few immediate consequences: $\overline{Z}'_{5,98}(M(6))\cong\overline{Z}_{5,98}(M(-3))$ is a singular projective variety; not all of its odd cohomology vanish; there is only one irreducible component with the maximal dimension 73. In general, we may not have a stratification of $Z'_{p,s}(M)$ (or $\overline{Z}'_{p,s}(M)$, $Z_{p,s}(M)$, $\overline{Z}_{p,s}(M)$) such that each stratum has only finite many (or no) fixed points under some $\mathbb{C}^*$-action. In other words, the existence of such a stratification for $r=2$ proved in  \cite[\S4]{CZ} seems to be very special for Kronecker quiver. 

The article is organized as follows.  In Section~\ref{Noncomm-section}, we recall the formula for noncommutative cluster variables given in \cite{ls-noncomm}. As a consequence, we obtain a formula for rank 2 quantum cluster variables in Section~\ref{quantum-section}. Theorem~\ref{mainthm2}(2) is proved in Section~\ref{proof_mainthm(2)}, and    proof of Theorem~\ref{mainthm2}(1)  is provided in Section~\ref{proof_mainthm(1)}.

\section{Noncommutative cluster variables}\label{Noncomm-section}
 In this section we briefly recall the formula for noncommutative cluster variables proved in \cite{ls-noncomm}. Fix a positive integer $r\geq 2$.
\begin{defn}\label{cn}
Let $\{c_n\}_{n\ge1}$ be the sequence  defined by the recurrence relation $$c_n=rc_{n-1} -c_{n-2},$$ with the initial condition $c_1=0$, $c_2=1$. 
\end{defn}
If $r=2$, then $c_n=n-1$. If $r>2$, let $\gamma_1=(r+\sqrt{r^2-4})/2$, $\gamma_2=(r-\sqrt{r^2-4})/2$, then 
$$
c_n= \frac{\gamma_1^{n-1}-\gamma_2^{n-1}}{\sqrt{r^2-4}}=\sum_{i\geq 0} (-1)^i \binom{n-2-i}{i}r^{n-2-2i}.
$$

In order to state our theorem, we fix an integer $n\geq 4$. For any pair $(a_1,a_2)$ of nonnegative integers, a Dyck path of type $a_1\times a_2$ is a lattice path from $(0,0)$ to $(a_1,a_2)$ that never goes above the diagonal joining $(0,0)$ and $(a_1,a_2)$. Among the Dyck paths of type $a_1\times a_2$ there is a unique maximal one denoted by $\mathcal{D}^{a_1\times a_2}$. 

\begin{defn}
Let $\mathcal{D}_n$ be the maximal Dyck path $\mathcal{D}^{(c_{n-1}-c_{n-2})\times c_{n-2}}$. It can be represented by a sequence $(w_0, \alpha_1,w_1,\cdots, \alpha_{c_{n-1}}, w_{c_{n-1}})$, where $w_0,\cdots,w_{c_{n-1}}$ are vertices and $\alpha_1,\cdots, \alpha_{c_{n-1}}$ are edges, such that $w_0=(0,0)$ and $w_{c_{n-1}}=(c_{n-1}-c_{n-2},c_{n-2})$ are the two endpoints of $\mathcal{D}_n$, and $\alpha_i $ connects $w_{i-1}$ and $w_i$ for $i=1,\dots,c_{n-1}$.
\end{defn}

\begin{rem}\label{Christoffel}
The word obtained from $\mathcal{D}_n$ by forgetting the vertices $w_i$ and replacing each horizontal edge by the letter $h$ and each vertical edge by the letter $v$ is (by definition) the Christoffel word of slope $c_{n-2}/(c_{n-1}-c_{n-2})$.
\end{rem}

\begin{defn}
Let $v_i$ be the upper end point of the $i$-th vertical edge of $\mathcal{D}_n$. More precisely, let $i_1<\cdots<i_{c_{n-2}}$ be the sequence of integers such that $\alpha_{i_j}$ is vertical for any $1\leq j\leq c_{n-2}$. Define a sequence $v_0,v_1,\cdots,v_{c_{n-2}}$ of vertices by $v_0=(0,0)$ and $v_j=w_{i_j}$. 
\end{defn}

\begin{exmp} Let $r=3$, $n=5$. The following figure illustrates edges $\{\alpha_i\}_{1\le i\le 8}$ and vertices $\{v_i\}_{0\le i\le 3}$ in $\mathcal{D}_5=\mathcal{D}^{(8-3)\times 3}$. 
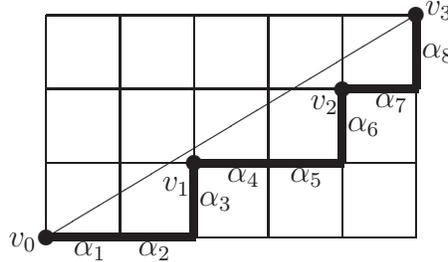
\begin{figure}[h]\setlength{\unitlength}{.7pt}
\begin{picture}(200,110)
\bigboxs{0,0}\bigboxs{40,0}\bigboxs{80,0}\bigboxs{120,0}\bigboxs{160,0}
\bigboxs{0,40}\bigboxs{40,40}\bigboxs{80,40}\bigboxs{120,40}\bigboxs{160,40}
\bigboxs{0,80}\bigboxs{40,80}\bigboxs{80,80}\bigboxs{120,80}\bigboxs{160,80}
\put(15,-10){$\tiny{\alpha_1}$}\put(50,-10){$\tiny{\alpha_2}$}
\put(132,30){$\tiny{\alpha_5}$}\put(98,30){$\tiny{\alpha_4}$}
\put(178,70){$\tiny{\alpha_7}$}
\put(83,17){$\tiny{\alpha_3}$}\put(163,57){$\tiny{\alpha_6}$}\put(203,97){$\tiny{\alpha_8}$}
\linethickness{1pt}\put(0,0){\line(5,3){200}}
\linethickness{3pt}\put(0,0){\line(1,0){81}}
\linethickness{3pt}\put(80,0){\line(0,1){41}}
\linethickness{3pt}\put(80,40){\line(1,0){81}}
\linethickness{3pt}\put(160,40){\line(0,1){41}}
\linethickness{3pt}\put(160,80){\line(1,0){41}}
\linethickness{3pt}\put(200,80){\line(0,1){41}}
\put(0,0){\circle*{8}}\put(80,40){\circle*{8}}\put(160,80){\circle*{8}}\put(200,120){\circle*{8}}
\put(-20,-5){$v_0$}\put(63,25){$v_1$}\put(143,69){$v_2$}\put(205,120){$v_3$}
\end{picture}
\end{figure}
\end{exmp}

We introduce certain special subpaths called colored subpaths. These colored subpaths are defined by certain slope conditions as follows.

\begin{defn}
For any $i<j$, let $s_{i,j}$ be the slope of the line through $v_i$ and $v_j$. Let $s$ be the slope of the diagonal, that is, $s=s_{0,c_{n-2}}$. 
\end{defn}

\begin{defn}[Colored subpaths]\label{alpha(i,k)}
For any $0\leq i<k\leq c_{n-2}$, let $\alpha(i,k)$ be the subpath of $\mathcal{D}_n$ defined as follows.

\noindent (1) If $s_{i,t}\leq s$ for all $t$ such that $i<t\leq k$, then let $\alpha(i,k)$ be the subpath from $v_i$ to $v_k$. Each of these subpaths will be called a BLUE subpath.

\noindent (2) If $s_{i,t}> s$ for some $i<t\leq k$, then

(2-a) if the smallest such $t$ is of the form $i+c_m-wc_{m-1}$ for some integers $3\le m\le n-1$ and $1\leq w< r-1$, then let $\alpha(i,k)$ be the subpath from $v_i$ to $v_k$. Each of these subpaths will be called a GREEN subpath. When $m$ and $w$ are specified, it will be said to be $(m,w)$-green.

(2-b) otherwise, let $\alpha(i,k)$ be the subpath from the immediate predecessor of $v_i$ to $v_k$. Each of these subpaths will be called a RED subpath.
\end{defn}

Note that every pair $(i,k)$ defines exactly one subpath $\alpha(i,k)$. We call these subpaths the \emph{colored subpaths} of $\mathcal{D}_n$. We denote the set of all these subpaths together with the single edges $\alpha_i$ by $\mathcal{P}(\mathcal{D}_n)$, that is,$$\mathcal{P}(\mathcal{D}_n)=\{\alpha(i,k)\,|\, 0\leq i<k\leq c_{n-2}\} \cup \{\alpha_1,\cdots,\alpha_{c_{n-1}} \}.$$

Now we define a set $\mathcal{F}(\mathcal{D}_n)$ of certain sequences of non-overlapping subpaths of $\mathcal{D}_n$. This set will parametrize the monomials in our expansion formula.

\begin{defn}
Let $\mathcal{F}(\mathcal{D}_n)$ be the set of sets 
$\{\beta_1,\cdots,\beta_t\}$ with variable size $t\ge 0$ satisfying the following conditions:
\begin{itemize}
\item $\beta_j\in \mathcal{P}(\mathcal{D}_n)$ for all $1\leq j\leq t$,
\item if $j\neq j'$, then $\beta_j$ and $\beta_{j'}$ have no common edge,
\item if $\beta_j=\alpha(i,k)$ and $\beta_{j'}=\alpha(i',k')$, then $i\neq k'$ and $i'\neq k$,
\item if $\beta_j$ is $(m,w)$-green, then at least one of the $(c_{m-1}-wc_{m-2})$ preceding edges of $v_i$ is contained in some $\beta_{j'}$.
\end{itemize}
\end{defn}

For each $\beta\in \mathcal{F}(\mathcal{D}_n)$, we say that $\alpha_i$ is \emph{supported on} $\beta$ if and only if $\alpha_i\in\beta$ or $\alpha_i$ is contained in some blue, green or red subpath $\beta_j\in\beta$. The \emph{support} of $\beta$, denoted by $\text{supp}(\beta)$, is defined to be the union of $\alpha_i$'s that are supported on $\beta$.

\begin{defn}
For each $\beta\in \mathcal{F}(\mathcal{D}_n)$ and each $i\in\{1,\cdots,c_{n-1}\}$, let
$$\beta_{[i]}=\left\{\begin{array}{ll}
x^{-1}y^{r}, &\text{ if }\alpha_i\text{ is not supported on }\beta\text{ and }\alpha_i\text{ is horizontal;}\\
&\\
x^{-1}y^{r-1}, &\text{ if }\alpha_i\text{ is not supported on }\beta\text{ and }\alpha_i\text{ is vertical;}\\ 
&\\
x^{-1}y^{0}, &\text{ if }\alpha_i\in\beta\text{ and }\alpha_i\text{ is horizontal;}\\
&\\
x^{-1}y^{-1}, &\text{ if }\alpha_i\in\beta\text{ and }\alpha_i\text{ is vertical;}\\
&\\
x^{0}y^{0},  &\text{ if }\alpha_i\text{ is horizontal and }\alpha_i\in\alpha(j,k)\in\beta\text{ for 
some }j,k;\\
&\\
x^{0}y^{-1}, &\text{ if }\alpha_i\text{ is vertical, }\alpha_{i-r+1}\text{ is horizontal, and }\alpha_i,\alpha_{i-r+1}\in\alpha(j,k)\in\beta\text{ for some }j,k;\\ 
&\\
x^{1}y^{-1}, &\text{ if }\alpha_i\text{ and }\alpha_{i-r+1}\text{ are  vertical, and }\alpha_i,\alpha_{i-r+1}\in\alpha(j,k)\in\beta\text{ for some }j,k;\\ 
&\\
x^{-1}y^{-1}, &\text{ if }\alpha_i\text{ is the first (vertical) edge of a red subpath }\alpha(j,k)\text{ in }\beta. 
\end{array}\right.$$
\end{defn}

Let $K$ be the skew field $\mathbb{C}(x,y)$. The Kontsevich automorphism $F_r:K\to K$ is given in \cite{K} by 
\begin{equation}\label{Kont_map}\left\{\begin{array}{l}x \mapsto xyx^{-1} \\ y \mapsto (1+y^r)x^{-1}.\end{array}\right.\end{equation}
Define $x_n=F_r^n(x)$. It has been shown that the expressions for $x_n$ in terms of $x$ and $y$ are Laurent polynomials  by Berenstein and Retakh \cite{BR} (including the case of unequal parameters) and earlier by Usnich \cite{U}.
 For $A_1,\dots, A_m\in K$, we let $\prod_{i=1}^m A_i$ denote the non-commutative product $A_1A_2\cdots A_m$.
 We are now ready to state the main result in \cite{ls-noncomm} (see \cite{r-noncomm} for a generalized formula in the case of unequal parameters).  

\begin{thm}\label{mainthm}
For any integer $n\geq 4$,
\begin{equation}\label{maineq1}x_{n-1}=\sum_{\mathbf{\beta}\in\mathcal{F}(\mathcal{D}_n)} xyx^{-1}y^{-1}x\left(\prod_{i=1}^{c_{n-1}}\beta_{[i]}\right)x^{-1}. 
\end{equation}
\end{thm}

We note that Di Francesco and Kedem \cite{DK, DK2} obtained a simpler formula for $x_n$ for the special case of $r=2$.

\section{Quantum cluster variables of rank 2}\label{quantum-section}
Quantum cluster algebras are introduced by Berenstein and Zelevinsky in \cite{BZ}.
Let $\mathcal{A}(r,r)$ be the quantum cluster algebra of rank 2, which is defined as follows.  The ambient field $\mathcal{F}$ is the skew-field of fractions
of the quantum torus with generators $X_1$ and $X_2$ satisfying the quasi-commutation relation $X_1X_2=qX_2X_1$. Then $\mathcal{A}(r,r)$ is the $\mathbb{Z}[q^{\pm 1/2}]$-subalgebra of $\mathcal{F}$ generated by a sequence of cluster variables $\{X_n : n \in \mathbb{Z}\}$ defined recursively from the relations
\begin{equation}\label{q-recursion}X_{n-1}X_{n+1}=q^{r/2}X_n^r+1, \quad \forall n\in\mathbb{Z}.\end{equation}
It is easy to check that
\begin{equation}\label{q-comm}
X_nX_{n+1}=qX_{n+1}X_n, \quad \forall n\in\mathbb{Z}. 
\end{equation}

We need the following relation between completely noncommutative cluster variables and quantum cluster variables. 
Let $R=\mathbb{Z}\langle x^{\pm1},y^{\pm1}\rangle$,  $\bar{R}=\mathbb{Z}[q^{\pm1/2}]\langle X_1^{\pm1},X_2^{\pm1}\rangle/(X_1X_2-qX_2X_1)\subset \mathcal{F}$. Let $\varphi:R\to\bar{R}$ be the ring homomorphism determined by $\varphi(x)=q^{1/2}X_1$, $\varphi(y)=q^{-1/2}X_2$. For convenience we also denote $\bar{z}=\varphi(z)$ for any $z\in R$. 

\begin{prop}\label{noncomm-quantum}
Let $n$ be a nonnegative integer. Then
$$\bar{x}_n=\varphi(x_n)=q^{1/2}X_{n+1}.$$
\end{prop}

\begin{proof}
We use induction on $n$. The base cases $n=0,1$ are easy to verify. For the inductive step, let $n\ge 1$ and suppose the equality holds for $n$ and $n-1$.
Note that
$$\aligned 
x_{n-1}x_{n+1}&\overset{(\ref{Kont_map})}{=}x_{n-1}x_{n}y_{n}x_{n}^{-1}\overset{(\ref{Kont_map})}=x_{n-1}x_{n}(1+y_{n-1}^r)x_{n-1}^{-1}x_{n}^{-1}\\
&\overset{(\ref{Kont_map})}=x_{n-1}x_{n}(1+(x_{n-1}^{-1}x_{n}x_{n-1})^r)x_{n-1}^{-1}x_{n}^{-1}=x_{n-1}x_{n}(1+x_{n-1}^{-1}x_{n}^rx_{n-1})x_{n-1}^{-1}x_{n}^{-1}\\
&=x_{n-1}x_{n}x_{n-1}^{-1}x_{n}^{-1}+x_{n-1} x_{n}x_{n-1}^{-1}x_{n}^{-1} x_{n}^r=x_{n-1}x_{n}x_{n-1}^{-1}x_{n}^{-1}(1+x_{n}^r).\endaligned$$
By inductive assumption, 
$$\aligned 
&\bar{x}_{n-1} \bar{x}_{n+1}= \bar{x}_{n-1} \bar{x}_{n}\bar{x}_{n-1}^{-1} \bar{x}_{n}^{-1} (1+ \bar{x}_{n}^r) \\
&= \left(q^{1/2}X_{n}\right)\left(q^{1/2}X_{n+1}\right) \left(q^{1/2}X_{n}\right)^{-1} \left(q^{1/2}X_{n+1}\right)^{-1}  (1+ \bar{x}_{n}^r)\overset{(\ref{q-comm})}=q(1+ \bar{x}_{n}^r). \endaligned$$
Then
$$
\bar{x}_{n+1}=q\bar{x}_{n-1}^{-1}(1+\bar{x}_n^r)=q^{1/2}X_n^{-1}(1+q^{r/2}X_{n+1}^r)=q^{1/2}X_{n+2}.
$$
This concludes the inductive proof.
\end{proof}

\begin{cor}\label{quantum-pos}
All coefficients of the quantum Laurent expression for $X_n$ in terms of $X_1$ and $X_2$ are nonnegative.
\end{cor}
\begin{proof}
It follows immediately from Theorem \ref{mainthm} and Proposition \ref{noncomm-quantum}.
\end{proof}

On the other hand, the quantum cluster variable $X_n$ of an acyclic quantum cluster algebra can be expressed using the quantum cluster character \cite{Q, r-quantum}. In particular, the following identity holds for the rank 2 case, 
$$
\aligned
X_n
&=\sum_{\mathbf{e}=(e_1,e_2)}q^{\frac{r(e_1^2+(d_2-e_2)^2)-d_1d_2}{2}}E(\Gr_\mathbf{e}M(n))X_1^{-d_1+r(d_2-e_2)}X_2^{re_1-d_2},
\endaligned
$$ where
$E(\Gr_\mathbf{e}M(n))= \sum_i (-1)^i\dim H^i(\Gr_\mathbf{e}M(n) )q^{ri/2}$ is the Serre polynomial.
Then Corollary~\ref{quantum-pos} implies the vanishing of the odd homologies of $\Gr_\mathbf{e}M(n)$, a fact that is also proved in \cite{Ef, KQ, N, Q}. As a consequence, $E(\Gr_\mathbf{e}M(n))=P_{\Gr_\mathbf{e}M(n)}(q^r)$, therefore the above identity can be rewritten as:
\begin{prop}\label{X_n-virtual}
$$
X_n
=\sum_{\mathbf{e}=(e_1,e_2)}q^{\frac{r(e_1^2+(d_2-e_2)^2)-d_1d_2}{2}}P_{\Gr_\mathbf{e}M(n)}(q^r)X_1^{-d_1+r(d_2-e_2)}X_2^{re_1-d_2}.
$$
\end{prop}

\section{Proof of Theorem~\ref{mainthm2}(2)}\label{proof_mainthm(2)}
Similarly to \cite[Proposition 3.8]{CZ}, we stratify the Grassmannian $\Gr_{e}(M)$ into the disjoint union of locally closed subvarieties 
$$\Gr_{p,e}(M)=\{(N_1,N_2)\in\Gr_e(M): N_1\in Z_{p,e_1}(M)\},\quad 0\le p\le d_2.$$  
It follows from the three properties of virtual Poincare polynomials that
$$P_{\Gr_e(M)}=\sum_p P_{\Gr_{p,e}(M)}=\sum_p\binom{p}{e_2-d_2+p}_qP_{Z_{p,e_1}(M)}=\sum_p\binom{p}{d_2-e_2}_qP_{Z_{p,e_1}(M)}.$$
We can also stratify $\Gr_e(M)$ into subvarieties
$$\Gr'_{p,e}(M)=\{(N_1,N_2)\in\Gr_e(M): N_2\in Z'_{p,d_2-e_2}(M)\},\quad 0\le p\le d_1.$$ 
A similar computation gives
$$P_{\Gr_{e}(M)}=\sum_p P_{\Gr'_{p,e}(M)}=\sum_p\binom{p}{e_1}_qP_{Z'_{p,d_2-e_2}(M)}.$$
Let $B$ be the $(d_1+1)\times (d_1+1)$ upper triangular matrix whose $(i,j)$-entry is $(-1)^{j-i}q^{\binom{j-i}{2}}\binom{j-1}{i-1}_q$ for $1\le i\le j\le d_1+1$. It is easy to verify that
$$B=
\begin{bmatrix} 
\binom{0}{0}_q&\binom{1}{0}_q&\cdots&\binom{d_1}{0}_q\\
0&\binom{1}{1}_q&\cdots&\binom{d_1}{1}_q\\
\vdots&\vdots&\ddots&\vdots\\
0&0&\cdots&\binom{d_1}{d_1}_q
\end{bmatrix}^{-1},
$$
and that 
$$\begin{bmatrix}
P_{Z'_{0,d_2-e_2}(M)}\\
\vdots\\
P_{Z'_{d_1,d_2-e_2}(M)}
\end{bmatrix}
=B
\begin{bmatrix}
P_{G_{0,e_2}(M)}\\
\vdots\\
P_{G_{d_1,e_2}(M)}
\end{bmatrix}.
$$
In particular,
$$P_{Z'_{p,d_2-e_2}(M)}=\sum_{e_1=p}^{d_1}(-1)^{e_1-p}q^{\binom{e_1-p}{2}}\binom{e_1}{p}_qP_{G_{e_1,e_2}(M)}.$$
Using the additive property of virtual Poincare polynomials, we have
$$\aligned
P_{\overline{Z}'_{p,d_2-e_2}(M)}
&=\sum_{p'=p}^{d_1}\sum_{e_1=p'}^{d_1}(-1)^{e_1-p'}q^{\binom{e_1-p'}{2}}\binom{e_1}{p'}_qP_{G_{e_1,e_2}(M)}\\
&=\sum_{e_1=p}^{d_1}P_{G_{e_1,e_2}(M)}\sum_{p'=p}^{e_1}(-1)^{e_1-p'}q^{\binom{e_1-p'}{2}}\binom{e_1}{p'}_q\\
&=\sum_{e_1=p}^{d_1}P_{G_{e_1,e_2}(M)}(-1)^{e_1-p}q^{\binom{e_1-p+1}{2}}\binom{e_1-1}{e_1-p}_q,\\
\endaligned
$$
where the last equality follows from an elementary computation:
$$\sum_{p'=p}^{e_1}(-1)^{e_1-p'}q^{\binom{e_1-p'}{2}}\binom{e_1}{p'}_q
=\sum_{i=0}^{e_1-p}(-1)^{i}q^{\binom{i}{2}}\binom{e_1}{i}_q=(-1)^{e_1-p}q^{\binom{e_1-p+1}{2}}\binom{e_1-1}{e_1-p}_q.$$
This concludes the proof of Theorem~\ref{mainthm2}(2).

\section{Proof of Theorem~\ref{mainthm2}(1)}\label{proof_mainthm(1)}



In order to prove Theorem~\ref{mainthm2}(1), we first compute the coefficient of $\bar{x}^{-d_1+r(d_2-e_2)}\bar{y}^{re_1-d_2}$ in $\bar{x}_5$. Multiplying the coefficient by $q^{\frac{1}{2}[-d_1+r(d_2-e_2)-(re_1-d_2)-1]}$, we get the coefficient of $X_1^{-d_1+r(d_2-e_2)}X_2^{re_1-d_2}$ in $X_6$, thanks to Proposition~\ref{noncomm-quantum}. Then we can compute $P_{\Gr_\mathbf{e}(M)}$ using Proposition~\ref{X_n-virtual}. 

In this section, all the computations take place in the quantum situation $\bar{R}$, hence we shall abuse notations by using $z\in R$ to mean $\varphi(z)\in \bar{R}$ and no confusion should arise. For example, $x$, $y$, $\beta_{[i]}$ actually mean $\bar{x}$, $\bar{y}$, $\varphi(\beta_{[i]})$ respectively. It is easy to verify the following identities:

(1) $x^ay^b=q^{ab}y^bx^a$, $y^bx^a=q^{-ba}x^ay^b$.

(2) $(x^ay^b)^i=q^{-ab\binom{i}{2}}x^{ai}y^{bi}$.

Theorem~\ref{mainthm} implies (remember that the computation takes place in $\bar{R}$)
$$x_{5}=qx\Big{(}\sum_{\beta\in F(\mathcal{D}_6)}\prod_{i=1}^{c_5}\beta_{[i]}\Big{)}x^{-1}.$$  Since $d_1=r^3-2r$, $d_2=r^2-1$, $e_2=1$, the monomial $x^{-d_1+r(d_2-e_2)}y^{re_1-d_2}$ is equal to $y^{re_1-r^2+1}$.  
If we plug in $q=1$, then we should get the commutative formula proved in \cite[Theorem 9]{LS}, which we shall briefly recall. For any $\mathbf{\beta}=\{\beta_1,\cdots,\beta_t\}$, let $|\mathbf{\beta}|_2$ be the total number of edges in $\beta_1,\cdots,\beta_t$, and $|\mathbf{\beta}|_1=\sum_{j=1}^t |\beta_j|_1$, where
$$
|\beta_j|_1=\left\{\begin{array}{ll} 0, &\text{ if }\beta_j=\alpha_i\text{ for some }1\leq i\leq c_{n-1} \\k-i, &\text{ if }\beta_j=\alpha(i,k)\text{ for some }0\leq i<k\leq c_{n-2}.  \end{array}   \right.
$$
Then after setting $q=1$, we have
$$x_{n-1}\big{|}_{q=1}=\sum_{\mathbf{\beta}\in\mathcal{F}(\mathcal{D}_n)}x^{r|\mathbf{\beta}|_1-c_{n-1}}y^{r(c_{n-1}-|\mathbf{\beta}|_2)-c_{n-2}}\big{|}_{q=1}.$$
In particular,
$$x_5\big{|}_{q=1}=\sum_{\mathbf{\beta}\in\mathcal{F}(\mathcal{D}_6)}x^{r|\mathbf{\beta}|_1-(r^3-2r)}y^{r(r^3-2r-|\mathbf{\beta}|_2)-(r^2-1)}\big{|}_{q=1}.$$
Therefore, those $\beta$ that contribute to the monomial $y^{re_1-r^2+1}$ must satisfy $|\beta|_1=r^2-2$, $|\beta|_2=r^3-2r-e_1$, and the first equality implies $|\beta|_1=c_4-1$, hence all but one of the vertical edges are in some color subpath in $\beta$. Thus we may assume 
\begin{equation}\label{eq:beta}
\beta=\{\alpha(0,\ell-1),\alpha(\ell,r^2-1)\}\cup\{\textrm{single edges}\}
\end{equation} 
for certain $1\leq \ell\leq r^2-1$. 

The following lemma determines the color of $\alpha(\ell,r^2-1)$. Its proof is easy and we omit.
\begin{lem}\label{lem:color}
 If $r>2$ then the subpath $\alpha(\ell,r^2-1)$ is
$$
 \left\{ 
  \begin{array}{l l}
    $(3,k)$\text{-green}, & \quad \text{if $\ell=jr+k$  for  $0\le j\le r-2$ and $1\leq k\leq r-2$};\\
    \\
    \text{red}, &\quad \text{if $\ell=jr+r-1$ for $0\le j\le r-2$, or $i=r^2-2$};\\
    \\
     $(4,j)$\text{-green}, &\quad \text{if $\ell=jr$ for $1\le j\le r-2$};\\
     \\
     $(3,k)$\text{-green}, & \quad \text{if $\ell=(r-1)r+k-1$  for  $1\leq k\leq r-2$};\\
     \\
     \text{empty}, & \quad \text{if $\ell=r^2-1$}.
  \end{array} \right.
$$
\end{lem}


First of all, the maximal Dyck path $\mathcal{D}_5$ is of size
$(c_4-c_3)\times c_3=(r^2-r-1)\times r$ and is represented by Christoffel word 
$(h^{r-1}v)^{r-1}h^{r-2}v$ (see Remark~\ref{Christoffel}); the maximal Dyck path $\mathcal{D}_6$ is of size
$(c_5-c_4)\times c_4=(r^3-r^2-2r+1)\times(r^2-1)$ and is represented by
$(A^{r-1}B)^{r-1}A^{r-2}B$ where $A=h^{r-1}v$ and $B=h^{r-2}v$. 

The rest of the proof is a case-by-case study depending on the integer $\ell$ in \eqref{eq:beta}. There are five cases to be considered.

\bigskip

\noindent\textbf{Case 1.} $\ell=jr+k$ where $1\le k\le r-1$, $0\le j\le r-2$. 

The subpath $\alpha(0,\ell-1)$ is blue and of the form $(A^{r-1}B)^jA^{k-1}$, with length $(r^2-1)j+r(k-1)$.
The subpath $\alpha(\ell,r^2-1)$ is either $(3,k)$-green (if $k<r-1$) or red (if $k=r-1$) according to Lemma \ref{lem:color}, and has form $A^{r-1-k}B(A^{r-1}B)^{r-j-2}A^{r-2}B$. Its length is
$r^3-r^2j-rk-2r+j$.

Then $\alpha(0,\ell-1)$ contributes 
$$\prod_{i=1}^{(r^2-1)j+r(k-1)}\beta_{[i]}=(y^{-(r-1)}xy^{-1})^j\cdot y^{-(k-1)}=(q^{r-1}xy^{-r})^j\cdot y^{-(k-1)}=q^{(r-1)j+rj(j-1)/2}x^jy^{-rj-k+1},$$
and similarly, $\alpha(jr+k,r^2-1)$ corresponds to
$$\aligned\prod_{i=(r^2-1)j+rk+1}^{r^3-2r}\beta_{[i]}
&=y^{-(r-1-k)}\cdot xy^{-1}\cdot (y^{-(r-1)}xy^{-1})^{r-j-2}\cdot y^{-(r-2)}\cdot xy^{-1}\\
&=q^{r^3/2-r^2j+rj^2/2+r^2/2-rk-rj/2+kj-r+j-1}x^{r-j}y^{-r^2+rj+k+1}.
\endaligned$$

In $\mathcal{D}_6 \setminus($the support of $\alpha(0,\ell-1)$ and $\alpha(\ell,r^2-1))$, there are $r$ edges.
The last vertical edge must be chosen (or is the first vertical edge of a red subpath).  If we go through all possible choices, then
$$\sum_{\beta} \prod_{i=(r^2-1)j+r(k-1)+1}^{(r^2-1)j+rk}\beta_{[i]}=x^{-(r-1)}\prod_{t=0}^{r-2}(q^{tr}y^r+1)\cdot x^{-1}y^{-1}.$$
Therefore
\begin{equation}\label{case1eq}\aligned 
&\sum_{\beta} \prod_{i=1}^{r^3-2r}\beta_{[i]}\\
&=
q^{r^3/2-r^2j+rj^2+r^2/2-rk+kj-r-1}x^jy^{-rj-k+1}x^{-(r-1)}\prod_{t=0}^{r-2}(q^{tr}y^r+1)\cdot x^{-1}y^{-1}x^{r-j}y^{-r^2+rj+k+1}\\
&=q^{r^3/2-r^2j+r^2/2-rk-1}y^{-r^2+1}\prod_{t=0}^{r-2}(q^{tr-r(r-j-1)}y^r+1).
\endaligned
\end{equation}

\noindent\textbf{Case 2.} $\ell=jr$ where $1\le j\le r-2$. 

The subpath $\alpha(0,\ell-1)$ is blue and its form is $(A^{r-1}B)^{j-1}A^{r-1}$.  Its length is $(r^2-1)j-r+1$.
The subpath $\alpha(\ell,r^2-1)$ is $(4,j)$-green according to Lemma  \ref{lem:color}. It is of the form
 $(A^{r-1}B)^{r-1-j}A^{r-2}B$ and of length $r^3-r^2j-2r+j$.
In $\mathcal{D}_6 \setminus($the support of $\alpha(0,\ell-1)$ and $\alpha(\ell,r^2-1))$, there are $r-1$ edges.
 Since $\alpha(\ell,r^2-1)$ is $(4,j)$-green, at least one of its preceding $c_3-jc_2=r-j$ edges is contained in $\beta$. 

Then $\alpha(0,\ell-1)$ contributes
$$
\aligned
\prod_{i=1}^{(r^2-1)j-r+1}\beta_{[i]}&=(y^{-(r-1)}xy^{-1})^{j-1}\cdot y^{-(r-1)}=(q^{r-1}xy^{-r})^{j-1}\cdot y^{-r+1}\\
&=q^{(r-1)(j-1)+r\binom{j-1}{2}}x^{j-1}y^{-rj+1}
=q^{(j-1)(rj/2-1)}x^{j-1}y^{-rj+1}.
\endaligned$$
Similarly, $\alpha(\ell,r^2-1)$ contributes
$$\aligned\prod_{i=(r^2-1)j+1}^{r^3-2r}\beta_{[i]}
&=(y^{-(r-1)}xy^{-1})^{r-1-j}\cdot y^{-(r-2)}xy^{-1}
=(q^{r-1}xy^{-r})^{r-1-j}\cdot q^{r-2}xy^{-r+1}\\
&=q^{(r-1)(r-1-j)+r\binom{r-1-j}{2}+r-2+r(r-1-j)}x^{r-j}y^{-r^2+rj+1}.
\endaligned$$
Considering admissible subsets of $\mathcal{D}_6 \setminus($the support of $\alpha(0,\ell-1)$ and $\alpha(\ell,r^2-1))$, we get
$$
\aligned
&\sum_\beta\prod_{i=(r^2-1)j-r+2}^{(r^2-1)j}\beta_{[i]}\\
&=(x^{-1}y^r+x^{-1})^{r-2}(x^{-1}y^{r-1}+x^{-1}y^{-1})-(x^{-1}y^r+x^{-1})^{j-1}(x^{-1}y^r)^{r-j-1}x^{-1}y^{r-1}\\
&=x^{-r+1}\Big{[}y^{-1}\prod_{t=0}^{r-2}(q^{tr}y^r+1)-q^{r\binom{r-j}{2}}y^{r^2-rj-1}\prod_{t=r-j}^{r-2}(q^{tr}y^r+1)\Big{]}.
\endaligned
$$

Then
$$
\aligned
&\sum_\beta\prod_{i=1}^{r^3-2r}\beta_{[i]}\\
&=q^{r^3/2-r^2j+rj^2+r^2/2-rj-r}x^{j-1}y^{-rj+1}x^{-r+1}
\\
&\cdot\Big{[}y^{-1}\prod_{t=0}^{r-2}(q^{tr}y^r+1)-q^{r\binom{r-j}{2}}y^{r^2-rj-1}\prod_{t=r-j}^{r-2}(q^{tr}y^r+1)\Big{]}x^{r-j}y^{-r^2+rj+1}\\
&=q^{r^3/2-r^2j+rj^2+r^2/2-rj-r}x^{j-1}y^{-rj+1}x^{-r+1}
\\
&\cdot\Big{[}q^{r-j}x^{r-j}y^{-r^2+rj}\prod_{t=0}^{r-2}(q^{r(t-r+j)}y^r+1)-q^{(r-j)(-r^2+rj-r+2)/2}x^{r-j}\prod_{t=r-j}^{r-2}(q^{r(t-r+j)}y^r+1)\Big{]}
\\
\endaligned
$$
\begin{equation}\label{case2eq}
\aligned
&=q^{r^3/2-r^2j+rj^2+r^2/2-rj-r}x^{j-1}y^{-rj+1}x^{-r+1}
\\
&\cdot x^{r-j}\Big{[}q^{r-j}y^{-r^2+rj}\prod_{t=0}^{r-2}(q^{r(t-r+j)}y^r+1)-q^{(r-j)(-r^2+rj-r+2)/2}\prod_{t=r-j}^{r-2}(q^{r(t-r+j)}y^r+1)\Big{]}
\\
&=q^{r^3/2-r^2j+r^2/2-1}y^{-rj+1}\Big{[}y^{-r^2+rj}\prod_{t=0}^{r-2}(q^{r(t-r+j)}y^r+1)-q^{-r\binom{r-j+1}{2}}\prod_{t=r-j}^{r-2}(q^{r(t-r+j)}y^r+1)\Big{]}.
\endaligned
\end{equation}

\bigskip

\noindent\textbf{Case 3.}  $\ell=(r-1)r$.

The subpath $\alpha(0,\ell-1)$ is blue and of the form $(A^{r-1}B)^{r-2}A^{r-1}$.  Its length is $(r^2-1)(r-2)+r(r-1)$.
The subpath $\alpha(\ell,r^2-1)$ is
either $(3,1)$-green if $r\ge 3$ or red if $r=2$. It is of the form $A^{r-2}B$, and of length $(r-2)r+(r-1)$.

Then $\alpha(0,\ell-1)$ contributes
$$
\aligned
\prod_{i=1}^{(r^2-1)(r-2)+r(r-1)}\beta_{[i]}
&=(y^{-(r-1)}xy^{-1})^{r-2}\cdot y^{-(r-1)}
=q^{(r+1)(r-2)^2/2}x^{r-2}y^{-r^2+r+1}.
\endaligned$$
Similarly $\alpha(\ell,r^2-1)$ contributes
$$\prod_{i=r^3-r^2-r+2}^{r^3-2r}\beta_{[i]}
=y^{-(r-2)}xy^{-1}=q^{r-2}xy^{-(r-1)}.
$$
In $\mathcal{D}_6 \setminus($the support of $\alpha(0,(r-1)r-1)$ and $\alpha((r-1)r,r^2-1))$, there are $r-1$ edges, whose corresponding word is $h^{r-2}v$. The last edge must be either chosen or the first edge of a red path.
$$
\aligned
\sum_\beta\prod_{i=(r^2-1)(r-2)+r(r-1)+1}^{r^3-r^2-r+1}\beta_{[i]}&=(x^{-1}y^r+x^{-1})^{r-2}x^{-1}y^{-1}=x^{-r+1}y^{-1}\prod_{t=1}^{r-2}(q^{tr}y^r+1)
\endaligned
$$
Then
\begin{equation}\label{case3eq}
\aligned
\sum_\beta\prod_{i=1}^{r^3-2r}\beta_{[i]}
&=q^{(r+1)(r-2)^2/2+r-2}x^{r-2}y^{-r^2+r+1}x^{-r+1}y^{-1}\prod_{t=1}^{r-2}(q^{tr}y^r+1)\cdot xy^{-(r-1)}\\
&=q^{-r^3/2+3r^2/2-1}y^{-r^2+1}\prod_{t=1}^{r-2}(q^{(t-1)r}y^r+1).
\endaligned
\end{equation}

\bigskip

\noindent\textbf{Case 4.}  $\ell=(r-1)r+k-1$ where $2\le k\le r-1$. 

The subpath $\alpha(0,\ell-1)$ is blue and of the form $(A^{r-1}B)^{r-1}A^{k-2}$.  Its length is $(r^2-1)(r-1)+r(k-2)$.
The subpath $\alpha(\ell,r^2-1)$ is
either $(3,k)$-green if $k\le r-2$ or red if $k=r-1$. It is of the form $A^{r-1-k}B$ and of length
 $(r-1-k)r+(r-1)$.
Then $\alpha(0,\ell-1)$ contributes
$$
\aligned
\prod_{i=1}^{(r^2-1)(r-1)+r(k-2)}\beta_{[i]}
&=(y^{-(r-1)}xy^{-1})^{r-1}\cdot y^{-(k-2)}
=q^{(r-1)(r^2/2-1)}x^{r-1}y^{-r^2+r-k+2}.
\endaligned$$
Similarly $\alpha(\ell,r^2-1)$ contributes
$$\prod_{i=(r^2-1)(r-1)+r(k-1)+1}^{r^3-2r}\beta_{[i]}
=y^{-(r-1-k)}xy^{-1}=q^{r-1-k}xy^{-(r-k)}.
$$
In $\mathcal{D}_6 \setminus($the support of $\alpha(0,(r-1)r+k-2)$ and $\alpha((r-1)r+k-1,r^2-1))$, there are $r$ edges, whose corresponding word is $h^{r-1}v$. The last edge must be either chosen or the first edge of a red path. 
$$
\aligned
\sum_\beta\prod_{i=(r^2-1)(r-1)+r(k-2)+1}^{(r^2-1)(r-1)+r(k-1)}\beta_{[i]}&=(x^{-1}y^r+x^{-1})^{r-1}x^{-1}y^{-1}=x^{-r}y^{-1}\prod_{t=1}^{r-1}(q^{tr}y^r+1)
\endaligned
$$
Then
\begin{equation}\label{case4eq}
\aligned
\sum_\beta\prod_{i=1}^{r^3-2r}\beta_{[i]}
&=q^{(r-1)(r^2/2-1)+r-1-k}x^{r-1}y^{-r^2+r-k+2}x^{-r}y^{-1}\prod_{t=1}^{r-1}(q^{tr}y^r+1)\cdot xy^{-(r-k)}\\
&=q^{-r^3/2+3r^2/2-rk+r-1}y^{-r^2+1}\prod_{t=0}^{r-2}(q^{tr}y^r+1).
\endaligned
\end{equation}

\bigskip

\noindent\textbf{Case 5.} $\ell=r^2-1$. 

In this case, the subpath $\alpha(\ell,r^2-1)$ vanishes. The subpath $\alpha(0,r^2-2)$ is blue and of the form
$(A^{r-1}B)^{r-1}A^{r-2}$. Its length is $(r^2-1)(r-1)+r(r-2)=r^3-3r+1$.
$$
\aligned
\sum_\beta\prod_{i=1}^{r^3-3r+1}\beta_{[i]}
&=(y^{-(r-1)}xy^{-1})^{r-1}y^{-(r-2)}=q^{(r-1)(r^2/2-1)}x^{r-1}y^{-r^2+2}.\\
\endaligned
$$
What is left in $\mathcal{D}_6$ contributes 
$$
\aligned
\sum_\beta\prod_{i=1r^3-3r+2}^{r^3-2r}\beta_{[i]}
&=(x^{-1}y^r+x^{-1})^{r-2}(x^{-1}y^{r-1}+x^{-1}y^{-1})=x^{-(r-1)}y^{-1}\prod_{t=0}^{r-2}(q^{tr}y^r+1).
\endaligned
$$
Then
\begin{equation}\label{case5eq}
\aligned
\sum_\beta\prod_{i=1}^{r^3-2r}\beta_{[i]}
&=q^{(r-1)(r^2/2-1)}x^{r-1}y^{-r^2+2}x^{-(r-1)}y^{-1}\prod_{t=0}^{r-2}(q^{tr}y^r+1)\\
&=q^{(r-1)(-r^2/2+1)}y^{-r^2+1}\prod_{t=0}^{r-2}(q^{tr}y^r+1).
\endaligned
\end{equation}

\noindent\textbf{Add up Cases 1--5.} Add (\ref{case1eq}), (\ref{case2eq}), (\ref{case3eq}), (\ref{case4eq}), (\ref{case5eq}), we get
\begin{eqnarray*}
&\quad&\sum_{k=1}^{r-1}\sum_{j=0}^{r-2}q^{r^3/2-r^2j+r^2/2-rk-1}y^{-r^2+1}\prod_{t=0}^{r-2}(q^{tr-r(r-j-1)}y^r+1)\\
&+&
\sum_{j=1}^{r-2}
q^{r^3/2-r^2j+r^2/2-1}y^{-rj+1}\Big{[}y^{-r^2+rj}\prod_{t=0}^{r-2}(q^{r(t-r+j)}y^r+1)-q^{-r\binom{r-j+1}{2}}\prod_{t=r-j}^{r-2}(q^{r(t-r+j)}y^r+1)\Big{]}\\
&+&
q^{-r^3/2+3r^2/2-1}y^{-r^2+1}\prod_{t=1}^{r-2}(q^{(t-1)r}y^r+1)+
\sum_{k=2}^{r-1}q^{-r^3/2+3r^2/2-rk+r-1}y^{-r^2+1}\prod_{t=0}^{r-2}(q^{tr}y^r+1)\\
&+&
q^{(r-1)(-r^2/2+1)}y^{-r^2+1}\prod_{t=0}^{r-2}(q^{tr}y^r+1),\\
&=&q^{\frac{(r-1)(2-r^2)}{2}}y^{-r^2+1}\bigg{(}
\sum_{k=1}^{r-1}\sum_{j=0}^{r-2}q^{r^3-r-r^2j-rk}\prod_{t=0}^{r-2}(q^{tr-r(r-j-1)}y^r+1)\\
&+&
\sum_{j=1}^{r-2}
q^{r^3-r-r^2j}y^{r^2-rj}\Big{[}y^{-r^2+rj}\prod_{t=0}^{r-2}(q^{r(t-r+j)}y^r+1)-q^{-r\binom{r-j+1}{2}}\prod_{t=r-j}^{r-2}(q^{r(t-r+j)}y^r+1)\Big{]} \quad  \quad \quad\\
&+&
q^{r^2-r}\prod_{t=1}^{r-2}(q^{(t-1)r}y^r+1)+
\sum_{k=2}^{r-1}q^{r^2-rk}\prod_{t=0}^{r-2}(q^{tr}y^r+1)+ \prod_{t=0}^{r-2}(q^{tr}y^r+1)\bigg{)},
\end{eqnarray*}
which equals to
\begin{equation}\label{prodbeta}
\aligned
&
q^{\frac{(r-1)(2-r^2)}{2}}y^{-r^2+1}\bigg{(}
(\sum_{k=0}^{r-2}q^{rk})\sum_{j=1}^{r-1}q^{r^2j}\prod_{t=0}^{r-2}(q^{tr-rj}y^r+1)\\
+&
\sum_{j=1}^{r-2}
q^{r^3-r-r^2j}\Big{[}\prod_{t=0}^{r-2}(q^{r(t-r+j)}y^r+1)-q^{-r\binom{r-j+1}{2}}y^{r^2-rj}\prod_{t=0}^{j-2}(q^{tr}y^r+1)\Big{]}\\
+&
q^{r^2-r}\prod_{t=0}^{r-3}(q^{tr}y^r+1)+
(\sum_{k=0}^{r-2}q^{rk})\prod_{t=0}^{r-2}(q^{tr}y^r+1)\bigg{)}.
\endaligned
\end{equation}

On the other hand, if we replace $X_1$ with $q^{-1/2}x$ and $X_2$ with $q^{1/2}y$ in Proposition~\ref{X_n-virtual} for $n=6$ and multiply by $q^{1/2}$, then  we obtain
\begin{equation}\label{eq:n=6}
\aligned
&q^{1/2}\sum_{\mathbf{e}=(e_1,e_2)}q^{\frac{r(e_1^2+(d_2-e_2)^2)-d_1d_2}{2}}P_{\Gr_\mathbf{e} M(6)}(q^r)(q^{-1/2}x)^{-d_1+r(d_2-e_2)}(q^{1/2}y)^{re_1-d_2},\\
\endaligned
\end{equation}
where $d_1=c_5=r^3-2r$, $d_2=c_4=r^2-1$. The sum of terms with  $e_2=1$ in \eqref{eq:n=6} is
$$
\aligned
&q^{1/2}\sum_{e_1}q^{\frac{r(e_1^2+(d_2-1)^2)-d_1d_2}{2}}P_{\Gr_\mathbf{e}M}(q^r)(q^{1/2}y)^{re_1-d_2}\\
&=\sum_{e_1}q^{\frac{(-r^2+2)(r+1)+re_1^2+re_1}{2}}P_{\Gr_\mathbf{e}M}(q^r)y^{re_1-d_2}
\endaligned
$$
Let $L$ be \eqref{prodbeta}. Proposition~\ref{noncomm-quantum} implies that
$$qxLx^{-1}=\sum_{e_1}q^{\frac{(-r^2+2)(r+1)+re_1^2+re_1}{2}}P_{\Gr_\mathbf{e}M}(q^r)y^{re_1-d_2},$$
which yields
$$\aligned L&=\sum_{e_1}q^{-1}x^{-1}q^{\frac{(-r^2+2)(r+1)+re_1^2+re_1}{2}}P_{\Gr_\mathbf{e}M}(q^r)y^{re_1-d_2}x\\
&=q^{\frac{(r-1)(2-r^2)}{2}}y^{-r^2+1}\sum_{e_1}q^{\frac{r(e_1^2-e_1)}{2}}P_{\Gr_\mathbf{e}M}(q^r)y^{re_1}.\endaligned$$

After replacing $q^r$ by $q$ and $y^r$ by $y$, we get
$$
\aligned
&\sum_{e_1}q^{e_1(e_1-1)/2}P_{Gr_\mathbf{e}M}(q)y^{e_1}
=\Big{(}\sum_{k=0}^{r-2}q^{k}\Big{)}\sum_{j=1}^{r-1}q^{rj}\prod_{t=0}^{r-2}(q^{t-j}y+1)\\
&\quad\quad\quad\quad+
\sum_{j=1}^{r-2}
q^{r^2-1-rj}\Big{[}\prod_{t=0}^{r-2}(q^{t-r+j}y+1)-q^{-\binom{r-j+1}{2}}y^{r-j}\prod_{t=0}^{j-2}(q^ty+1)\Big{]}\\
&\quad\quad\quad\quad+
q^{r-1}\prod_{t=0}^{r-3}(q^ty+1)+
(\sum_{k=0}^{r-2}q^k)\prod_{t=0}^{r-2}(q^ty+1).
\endaligned
$$

It is straightforward to check the following identity: for any integers $a\leq b$, we have
$$\prod_{i=a}^b(q^iy+1)=\sum_{i} q^{i(i-1)/2+ai}\binom{b-a+1}{i}_q\; y^i.$$
The identity implies that
$q^{e_1(e_1-1)/2}P_{Gr_eM}(q)$ is equal to
$$\aligned &[r-1]_q\sum_{j=1}^{r-1}q^{rj}q^{e_1(e_1-1)/2-je_1}\binom{r-1}{e_1}_q\\
&+\sum_{j=1}^{r-2}q^{r^2-1-rj}\bigg{[} q^{e_1(e_1-1)/2+(-r+j)e_1}\binom{r-1}{e_1}_q-q^{-\binom{r-j+1}{2}}q^{(e_1-r+j)(e_1-r+j-1)/2}\binom{j-1}{e_1-r+j}_q\bigg{]}\\
&+q^{r-1}q^{e_1(e_1-1)/2}\binom{r-2}{e_1}_q+[r-1]_qq^{e_1(e_1-1)/2}\binom{r-1}{e_1}_q.\endaligned$$

Canceling $q^{e_1(e_1-1)/2}$ out, we get
$$
\aligned P_{Gr_eM}(q)=&[r-1]_q\sum_{j=1}^{r-1}q^{(r-e_1)j}\binom{r-1}{e_1}_q+\sum_{j=1}^{r-2}q^{r^2-1-rj+(-r+j)e_1}\bigg{[} \binom{r-1}{e_1}_q-\binom{j-1}{e_1-r+j}_q\bigg{]}\\
& +q^{r-1}\binom{r-2}{e_1}_q+[r-1]_q\binom{r-1}{e_1}_q\\
\endaligned,$$
Hence
\begin{eqnarray*}
P_{Gr_eM}(q)&\overset{(*)}=&[r-1]_q\sum_{j=0}^{r-1}q^{(r-e_1)j}\binom{r-1}{e_1}_q+\sum_{j=1}^{r-2}q^{r^2-1-rj+(-r+j)e_1}\bigg{[} \binom{r-1}{e_1}_q-\binom{j-1}{e_1-r+j}_q\bigg{]}\\
&&\quad + q^{r-1-e_1}(\binom{r-1}{e_1}_q-\binom{r-2}{e_1-1}_q)\\
&=&[r-1]_q\sum_{j=0}^{r-1}q^{(r-e_1)j}\binom{r-1}{e_1}_q+\sum_{j=1}^{r-1}q^{r^2-1-rj+(-r+j)e_1}\bigg{[} \binom{r-1}{e_1}_q-\binom{j-1}{e_1-r+j}_q\bigg{]}\\
&=&[r-1]_q\sum_{j=0}^{r-1}q^{(r-e_1)j}\binom{r-1}{e_1}_q+\sum_{j=1}^{r-1}q^{(r-e_1)(r-j)-1}\bigg{[} \binom{r-1}{e_1}_q-\binom{j-1}{e_1-r+j}_q\bigg{]}\\
&=&[r-1]_q\sum_{j=0}^{r-1}q^{(r-e_1)j}\binom{r-1}{e_1}_q+\sum_{j=1}^{r-1}q^{(r-e_1)j-1}\bigg{[} \binom{r-1}{e_1}_q-\binom{r-j-1}{e_1-j}_q\bigg{]}\\
&=&[r-1]_q\binom{r-1}{e_1}_q + \sum_{j=1}^{r-1}q^{(r-e_1)j-1}\bigg{[} [r]_q\binom{r-1}{e_1}_q-\binom{r-j-1}{e_1-j}_q\bigg{]}, \\
\end{eqnarray*}
where $(*)$ follows from the identity 
$q^{-i}(\binom{r-1}{i}_q-\binom{r-2}{i-1}_q)=\binom{r-2}{i}_q$.
Since $[n]_q=\binom{n}{1}_q$ for every integer $n$, the proof of Theorem~\ref{mainthm2}(1) is thus completed.

\end{document}